\documentclass{article}
\usepackage{amsmath}
\usepackage{amsfonts}
\usepackage{amssymb}
\usepackage{amsthm}
\usepackage{graphicx}
\usepackage{fancyhdr}
\usepackage{geometry}

\makeatletter
\DeclareRobustCommand{\format@sec@number}[2]{{\normalfont\upshape#1}#2}

\makeatother

\def\e{\varepsilon}

\def\a{\alpha}

\def\d{\delta}

\def\l{\lambda}

\def\R{\mathbb R}

\def\Z{\mathbb Z}
\def\T{\mathbb T}

\def\C{\mathbb C}

\def\Q{\mathbb Q}

\def\({\biggl(}
\def\){\biggr)}
\def\<{\mathbf{\langle}}
\def\>{\mathbf{\rangle}}

\numberwithin{equation}{section}

\newtheorem{theorem}[equation]{Theorem}
\newtheorem{proposition}[equation]{Proposition}
\newtheorem{lemma}[equation]{Lemma}

\theoremstyle{definition}

\theoremstyle{definition}

\theoremstyle{remark}
\newtheorem*{remark}{Remark}

\title{Non-standard real-analytic realizations of some rotations of the circle}
\author{Shilpak Banerjee}

\begin{document}

\maketitle

\begin{abstract}
We extend some aspects of the smooth approximation by conjugation method to the real-analytic set-up and create examples of zero entropy, uniquely ergodic real-analytic diffeomorphisms of the two dimensional torus metrically isomorphic to some (Liouvillian) irrational rotations of the circle.  
\end{abstract}

\section{Introduction}

A smooth ($C^\infty$) diffeomorphism $f$ of a compact manifold $M$ preserving a smooth measure $m$ is said to be a \emph{smooth realization} of an abstract system $(X,T,\nu)$ if there exists a metric isomorphism between the two systems. This problem of smooth realization remains stubbornly intractable in general, despite the fact that after assuming finiteness of the measure theoretic entropy of the system $(X,T,\nu)$ and in some cases large ($>2$) dimension of $M$, there is no known restriction for $(X,T,\nu)$ to be realized as a smooth dynamical system on $M$. See section 7.2 in \cite{FK} for a detailed discussion of this question.

Several variants of the smooth realization problem described above exists and in this paper we are interested in a particular case, namely the problem of \emph{non-standard realization} in the real-analytic set-up. The problem of non-standard realization starts with a diffeomorphism $g$ of a compact manifold $N$ preserving a smooth measure $\nu$ and seeks out a diffeomorphism $f$ of a compact manifold $M$ preserving a smooth measure $m$ such that the two systems are metrically isomorphic but not smoothly isomorphic. In this case we say that $f$ is a \emph{non-standard realization} of $g$ on the manifold $M$.

The earliest examples of non-standard realization appears in \cite{AK} where D.V. Anosov and A. Katok developed a scheme known as the \emph{approximation by conjugation} method. They used this scheme to produce examples of non-standard realization of some irrational rotations of the circle as an ergodic smooth diffeomorphism of any manifold $M$ preserving a smooth measure $m$ provided the manifold admits an effective smooth action of the circle by $m$-preserving smooth diffeomorphisms. The diffeomorphisms constructed in \cite{AK} realized circle rotations with Liouvillan rotation numbers. However it was not clear from this construction which Liouvillian rotations were realized. Later B. Fayad, M. Saprykina and A. Windsor extended this result in \cite{FSW} and proved that \emph{any} Liouvillian rotation of the circle can be realized. Moreover they also proved that if $M$ is a finite dimensional torus, then this realization can be made uniquely ergodic. Proving unique ergodicity using this technique brings out another aspect of this theory since one needs to retain control over every orbit as opposed to almost every orbit. 

Also it should be mentioned that in \cite{AK}, examples of non-standard realization of some ergodic translations of any finite dimensional torus were produced on any manifold admitting an effective smooth action of the circle by $m$-preserving diffeomorphisms. Recently, M. Benhenda in \cite{Mb-ts} expanded this result to the case of ergodic translations by vectors with \emph{one arbitrary} Liouvillian coordinate on any finite dimensional torus.  

To complete the history of the smooth realization problem and its variants, it should also be noted that recently M. Benhenda in \cite{Mb-gk} produced examples of smooth diffeomorphisms on the product of a circle with a Hilbert cube having perfect Kronecker sets as their spectrum.     

In our paper we seek to use the \emph{approximation by conjugation} method and find examples of \emph{real-analytic} uniquely ergodic diffeomorphisms on the two dimensional torus that are metrically isomorphic to some irrational rotation of the circle. However requiring the diffemorphism to be real-analytic imposes some additional difficulties (see section 7.2 in \cite{FK}). One possible way to overcome such difficulties is to work in manifolds which have a large collection of real-analytic diffeomorphisms commuting with the circle action and their singularities are uniformly bounded away from a complex neighbourhood of the real domain. An example of such a construction appears in \cite{FK-wm}. In our case we work in the torus and customize the approximation by conjugation scheme with a trick that appeared in \cite{BK}. The key idea is to use entire functions that approximate certain carefully chosen step functions.

Our main theorem can be stated as follows:

\begin{theorem}
There exists uniquely ergodic real-analytic diffeomorphisms of the two dimensional torus $\T^2$ preserving the Lebesgue measure that are metrically isomorphic to some irrational rotations of the circle. 
\end{theorem}
In this article we produce a complete proof of the above theorem. However we would also like to point out that the result easily generalizes to a higher dimensional torus. In section \ref{section_high_dimension}, we outline a proof of the above theorem in the case of a torus of any dimension. 

We would also like to point out that very recently P. Kunde has proved the existence of a Real-analytic weak mixing diffeomorphism preserving a measurable Riemannian metric in \cite{Ku} using some of the techniques developed in this paper. This solves Problem 3.9 posed by R. Gunesch and A. Katok in \cite{GK}.


\section{Preliminaries} \label{section preliminaries}

The purpose of this section is to introduce some notations, give a brief description of the topology of real-analytic diffeomorphisms and finally lay down the approximation by conjugation scheme customized for our problem.

Throughout this paper, we will denote the two dimensional torus by $\T^2:=\R^2/\Z^2$. We will use $\mu$ to denote the standard Lebesgue measure on $\T^2$ and $\l$ to denote the Lebesgue measure on the unit circle $\T^1:=\R/\Z$.

\subsubsection*{The topology of real-analytic diffeomorphisms}

We start this section by giving a description of the space of measure preserving real-analytic diffeomorphisms on $\T^2$. 

Any real-analytic diffeomorphism on $\T^2$ homotopic to the identity admits a lift to a map from $\R^2$ to $\R^2$ of the form $F(x_1,x_2)=(x_1+f_1(x_1,x_2),x_2+f_2(x_1,x_2))$, where $f_i:\R^2\to \R^2$ is a $\Z^2$-periodic real-analytic functions. Any real-analytic $\Z^2$-periodic function on $\R^2$ can be extended as a complex analytic function using the natural embedding $(x_1,x_2)\mapsto (x_1+iy_1,x_2+iy_2)$ of $\R^2$ in $\C^2$. For a fixed $\rho>0$, let $\Omega_\rho:=\{(z_1,z_2)\in\C^2:|\text{Im}(z_1)|<\rho \text{ and  }|\text{Im}(z_2)|<\rho\}$ and for a function $f$ defined on this set, define $\|f\|_\rho:=\sup_{\Omega_\rho}|f((z_1,z_2))|$.  We define $C^\omega_\rho(\T^2)$ to be the space of all $\Z^2$-periodic real-analytic functions on $\R^2$ that extends to a holomorphic function on $\Omega_\rho$ and $\|f\|_\rho<\infty$.

Now we define, $\text{Diff }^\omega_\rho(\T^2,\mu)$ to be the set of all measure preserving real-analytic diffeomorphisms of $\T^2$ homotopic to the identity, whose lift $F$ to $\R^2$ satisfies $f_i\in C^\omega_\rho(\T^2)$.

The metric in $\text{Diff }^\omega_\rho(\T^2,\mu)$ is defined by 
\begin{align*}
d_\rho(F,G)=\max_{i=1,2}\{\inf_{n\in\Z}\|f_i-g_i+n\|_\rho\}
\end{align*}
Let $(F_1,F_2)$ be the lift of a diffeomorphism $F\in\text{Diff }^\omega_\rho(\T^2,\mu)$, we define 
\begin{align*}
\|DF\|_\rho:=\max_{\substack{i=1,2\\j=1,2}}\Big\|\frac{\partial F_i}{\partial x_i}\Big\|_\rho
\end{align*}

This completes the description of the analytic topology necessary for our construction. Also throughout this paper, the word ``diffeomorphism" will refer to a real-analytic diffeomorphism. Also, the word ``analytic topology" will refer to the topology of $\text{Diff }^\omega_\rho(\T^2,\mu)$ described above. See \cite{S} for a more extensive treatment of these spaces.

\subsubsection*{The approximation by conjugation scheme}

Next in this section, we give a brief description of the approximation by conjugation scheme tailored for our purpose. This scheme was introduced in \cite{AK}. For a relatively upto date description of this scheme and its usefulness, one might refer to \cite{FK}.

We start with $\phi$, a measure preserving $\T^1$ action on the torus $\T^2:=\R^2/\Z^2$ defined as follows:
\begin{align*}
\phi^t\big((x_1,x_2)\big)=(x_1+t,x_2)
\end{align*}
The required measure preserving ergodic diffeomorphism $T$ is going to be constructed as a limit of periodic measure preserving diffeomorphisms $T_n$ in the analytic topology described above. The diffeomorphisms $T_n$ are defined by 
\begin{align*}
T_n:=H_n^{-1}\circ\phi^{\a_n}\circ H_n
\end{align*}
where $\a_n\in \Q\cap [0,1)$ and $H_n\in \text{Diff }^\omega_\rho(\T^2)$. The diffeomorphisms $H_n$ and the rationals $\a_n$ s are constructed inductively. Given $H_n$ and $\a_n$, we construct at the $n+1$ th stage, a diffeomorphism $h_{n+1}\in  \text{Diff }^\omega_\rho(\T^2)$ and define:
\begin{align*}
H_{n+1}=h_{n+1}\circ H_n
\end{align*}
$\a_{n+1}$ is then defined and we would require it to be extremely close to $\a_n$ in order to ensure the convergence of our construction.

The construction of $h_n$ to suit our purpose is done in section \ref{hn}. At each stage we ensure that $T_n$ satisfies a finite version of the conjugacy that we would need eventually to conclude existence of the required metric isomorphism for the limit diffeomorphism. The measure theoretic result we require in that direction appears in section \ref{mt}.


\section{A measure theoretic lemma}\label{mt}

In this section, our goal is to prove an abstract lemma from measure theory which is a slight generalization of Lemma 4.1 in \cite{AK}.

First we recall a few definitions. A sequence of partitions $\{\mathcal{P}_n\}_n$ of a \emph{Lebesgue space}\footnote{Also known as a \emph{standard probability space} or a \emph{Lebesgue-Rokhlin space}. We consider those spaces which are isomorphic mod $0$ to the unit interval with the usual Lebesgue measure.} $(M,\mu)$ is called \emph{generating} if there exists a measurable subset $M'$ of full measure such that $\{x\}=\cap_{n=1}^\infty\mathcal{P}_n(x)\;\;\forall x\in M'$. We say that the sequence $\{\mathcal{P}_n\}_n$ is \emph{monotonic} if $\mathcal{P}_{n+1}$ is a refinement of $\mathcal{P}_n$. 

Lemma 4.1 from \cite{AK} gave us an easily checkable finite version of the conjugacy that one can use to prove the existence of a metric isomorphism of the limiting diffeomorphisms. Since the generating partitions used in the $C^\infty$ non-standard realization problem can easily made to be monotonic, this Lemma was sufficient. But in the real-analytic case, our construction is not flexible enough to guarantee monotonicity, so we need a modified version. Let us recall Lemma 4.1 from \cite{AK} since we will need it for our version.

\begin{lemma} \label{4.1}
Let $\{M^{(i)},\mu^{(i)}\},\;i=1,2$ be two Lebesgue spaces. Let $\mathcal{P}_n^{(i)}$ be a \emph{monotonic} sequences of generating finite partitions of $M^{(i)}$. Let $T_n^{(i)}$ be a sequence of automorphisms of $M^{(i)}$ satisfying $T_n^{(i)}\mathcal{P}_n^{(i)}=\mathcal{P}_n^{(i)}$ and suppose $ \lim_{n\to\infty}T_n^{(i)}=T^{(i)}$ weakly.
Suppose $\exists$  metric isomorphisms $K_n:M^{(1)}/\mathcal{P}_n^{(1)}\to M^{(2)}/\mathcal{P}_n^{(2)}$ satisfying:
\begin{align}
& K_n^{-1}T_n^{(2)}|_{\mathcal{P}_n^{(2)}}K_n=T_n^{(1)}|_{\mathcal{P}_n^{(1)}}\\
& K_{n+1}(\mathcal{P}_{n}^{(1)})=K_n(\mathcal{P}_{n}^{(1)})
\end{align}
Then the automorphisms $T^{(1)}$ and $T^{(2)}$ are metrically isomorphic.
\end{lemma}
We would also like to point out at this point of time that $K$ in the proof was defined to be
\begin{align}\label{K(x)}
K(x):=\cap_{n=1}^\infty K_n(P_n^{(1)}(x))\qquad\text{a.e.}\;\; x\in M^{(1)}
\end{align} 

Now we prove this variation which will allow us to accommodate a marginal ``twist" that will appear in our construction.

\begin{lemma}\label{mtl}
Let $\{M^{(i)},\mu^{(i)}\},\;i=1,2$ be two Lebesgue spaces and let $\mathcal{P}_n^{(1)}$. Let $\mathcal{P}_n^{(i)}$ be a sequence of generating finite partitions of $M^{(i)}$. Let $\{\e_n\}$ be a sequence of positive numbers satisfying $\sum_{n=1}^\infty\e_n<\infty$. In addition, assume that there exists a sequence of sets $\{E_n^{(i)}\}$ in $M^{(i)}$ satisfying:
\begin{align}
& \mu^{(1)}(E_n^{(1)})=\mu^{(2)}(E_n^{(2)})<\e_n\label{mtl 1}\\
& P_{n+1}^{(1)}(x)\setminus E_{n+1}^{(1)}\subset P_{n}^{(1)}(x)\quad\forall\; x\in M^{(1)}\setminus E_{n+1}^{(1)}\label{mtl 2}\\
& P_{n+1}^{(2)}(y)\subset P_{n}^{(2)}(y)\quad \forall\; y\in M^{(2)}\label{mtl 3}
\end{align}
 Let $T_n^{(1)}$ and $T_n^{(2)}$ be two sequences of automorphisms of the spaces $M^{(1)}$ and $M^{(2)}$ satisfying:
\begin{align}
& T_n^{(i)}\mathcal{P}_n^{(i)}=\mathcal{P}_n^{(i)}\quad\quad i=1,2\label{mtl 4}\\
& \lim_{n\to\infty}T_n^{(i)}=T^{(i)}\quad\quad i=1,2\label{mtl 5}\\
& T_n^{(i)}(\cup_{m=n}^\infty E_m^{(i)})=\cup_{m=n}^\infty E_m^{(i)}\quad\quad i=1,2\label{mtl 6}
\end{align}
Note that the limit in \ref{mtl 5} is taken in the weak topology. Suppose additionally there exists a sequence of metric isomorphisms $K_n:M^{(1)}/\mathcal{P}_n^{(1)}\to M^{(2)}/\mathcal{P}_n^{(2)}$ satisfying:
\begin{align}
& K_n^{-1}T_n^{(2)}|_{\mathcal{P}_n^{(2)}}K_n=T_n^{(1)}|_{\mathcal{P}_n^{(1)}}\label{mtl 7}\\
& K_{n+1}(\mathcal{P}_{n+1}^{(1)}(x))\subset K_{n}(\mathcal{P}_{n}^{(1)}(x))\quad\forall\; x\in M^{(1)}\setminus E_{n+1}^{(1)}\label{mtl 8}
\end{align}
Then the automorphisms $T^{(1)}$ and $T^{(2)}$ are metrically isomorphic.
\end{lemma}

\begin{proof}
Put $F^{(i)}_N:=\cup_{n=N}^\infty E_{n}^{(i)}$. Consider the sequence of Lebesgue spaces $M^{(i)}_N:=M^{(i)}\setminus F^{(i)}_N$.

\vspace{2mm}

\noindent \emph{Claim 1: $\exists$ a metric isomorphism $K_{(N)}:M^{(1)}_N\to M^{(2)}_N$, satisfying $K_{(N)}^{-1}T^{(2)}|_{M_N^{(2)}}K_{(N)}=T^{(1)}|_{M_N^{(1)}}$.}

We define $\mathcal{P}^{(i)}_{N,k}$, a finite measurable partition of $M^{(i)}_N$ by $\mathcal{P}^{(i)}_{N,k}(x):=\mathcal{P}^{(i)}_{N+k}(x)\setminus F^{(i)}_N$. We note that the sequence of partition $\{\mathcal{P}^{(i)}_{N,k}\}_{k}$ is generating because $\{\mathcal{P}^{(i)}_{k}\}_k$ is generating. Additionally, condition \ref{mtl 2} makes $\{\mathcal{P}^{(i)}_{N,k}\}_k$ a monotonic sequence of partition. We define $K_{N,k}(\mathcal{P}^{(1)}_{N,k}(x)):=K_{N+k}(\mathcal{P}^{(1)}_{N+k}(x))\setminus F^{(2)}_{N}$. 
With this definition we claim that $K_{N,k+1}(\mathcal{P}^{(1)}_{N,k}(x))=K_{N,k}(\mathcal{P}^{(1)}_{N,k}(x))$ $\forall x\in M^{(1)}_N$. (Indeed, from \ref{mtl 8} we get for a.e. $x\in M^{(1)}_N$, $K_{N,k+1}(\mathcal{P}^{(1)}_{N,k+1}(x))=K_{N+k+1}(\mathcal{P}^{(1)}_{N+k+1}(x))\setminus F^{(2)}_{N}\subset K_{N+k}(\mathcal{P}^{(1)}_{N+k}(x))\setminus F^{(2)}_N=K_{N,k}(\mathcal{P}^{(1)}_{N,k}(x))$. This with the fact that $K_{N,k+1}(\mathcal{P}^{(1)}_{N,k+1}(x))\in \mathcal{P}^{(2)}_{N,k+1}$ and $\{\mathcal{P}^{(2)}_{N,k}\}_k$ is a monotonic sequence of partitions helps us in concluding the claim). 
Observe that \ref{mtl 4}, \ref{mtl 6} and \ref{mtl 7} guarantees that $K_{N,k}^{-1}T_{N+k}^{(2)}|_{\mathcal{P}_{N,k}^{(2)}}K_{N,k}=T_{N+k}^{(1)}|_{\mathcal{P}_{N,k}^{(1)}}$. So we can apply Lemma \ref{4.1} to guarantee the existence of a metric isomorphism $K_{(N)}:M^{(1)}_N\to M^{(2)}_N$ defined for a.e. $x\in M^{(1)}_N$ by $K_{(N)}(x)=\cap^\infty_{k=1} K_{N,k}(\mathcal{P}^{(1)}_{N,k}(x))$. This finishes claim 1. 

\vspace{2mm}

\noindent\emph{Claim 2: $K_{(N+1)}(x)=K_{N}(x)$ for a.e. $x\in M^{(1)}_{N}$}

Follows from the definition of $K_{(N)}$. Indeed, note that $K_{(N+1)}(x)=\cap^\infty_{k=1} K_{N+1,k}(\mathcal{P}^{(1)}_{N+1,k}(x))=\cap^\infty_{k=1}K_{N+k+1}(\mathcal{P}^{(1)}_{N+k+1}(x))\setminus F^{(2)}_{N+1}=\cap^\infty_{k=0} K_{N+k+1}(\mathcal{P}^{(1)}_{N+k+1}(x))\setminus F^{(2)}_{N+1}$. The last equality follows from \ref{mtl 8}.

\vspace{2mm}

\noindent\emph{Claim 3: There exists a metric isomorphism $K:M^{(1)}\to M^{(2)}$ satisfying $K^{-1} T^{(2)}  K=T^{(1)}$}

Note that condition \ref{mtl 1} implies that a.e. $x\in E_n^{(1)}$ for at most finitely many $n$. Indeed, if $E=\{x:x\in E_n\text{ for infinitely many } n\}$, then $E\subset\cap_{n=m}^\infty E_n\;\forall\; m$ and $\lim_{m\to\infty}\mu^{(1)}(\cap_{n=m}^\infty E_n)=0$. So we can define for a.e. $x\in M^{(1)}$, $K(x):=K_{(N)}(x)$ if $x\in M^{(1)}_N$. Now claim 3 easily follows from claim 2.
\end{proof}


\section{Construction of $h_{n+1}$ and the conjugating analytic diffeomorphism}\label{hn}

As mentioned earlier, the construction in an inductive process so we assume that we have already defined $h_1,\ldots ,h_n\in\text{ Diff }^\omega_\rho(\T^2)$, two sequences of integers $\{k_m\}_{m=0}^{n-1}$ and $\{l_m\}_{m=0}^{n-1}$ and a sequence of rational numbers $\{\a_m\}_{m=1}^n$. At the $n+1$ th stage we construct  $h_{n+1}\in\text{ Diff }^\omega_\rho(\T^2)$ as a composition of four different diffeomorphisms. 

First, we let $l_n$ to be a large enough integer so that in particular we have 
\begin{align}\label{ln}
l_n>2^n\|DH_n\|_0
\end{align}
We will use this parameter to ensure that a sequence of partitions we construct later is generating. 

At this point we continue with our construction by assuming $k_n$ to be any integer. Later we will impose more conditions on the largeness of $k_n$ (see \ref{k_n restriction}) to ensure that $\a_{n+1}$ is sufficiently close to $\a_n$ which will be required to show the convergence of $T_n$. Largeness of $\{k_n\}$ will also be during the proof of unique ergodicity of $T$ (see proof of proposition \ref{longest}). So later we will choose $k_n$ to be larger than the maximum of these two requirements.

The key idea which we use here for constructing the diffeomorphism $h_{n+1}$ is from \cite{BK}. We observe that using functions of the form $(x_1,x_2)\mapsto (x_1+\tilde{\psi}_1(x_2),x_2)$ and $(x_1,x_2)\mapsto (x_1,x_2+\tilde{\psi}_2(x_1))$, where $\tilde{\psi}_i$s are carefully chosen step functions, we can ``break" a partition of $\T^2$ which is permuted cyclically by $\phi^{\a_{n+1}}$ and reform it into a generating partition that is permuted cyclically by $T_{n+1}$. One can imagine that this construction is similar to playing the sliding puzzle game called ``mystic square" or ``game of 15".  

\subsubsection*{Construction of a discontinuous version of the conjugating diffeomorphism}

\begin{figure}
\centering
\begin{tabular}{c c}
 & \raisebox{-2cm}{\includegraphics[scale=.1]{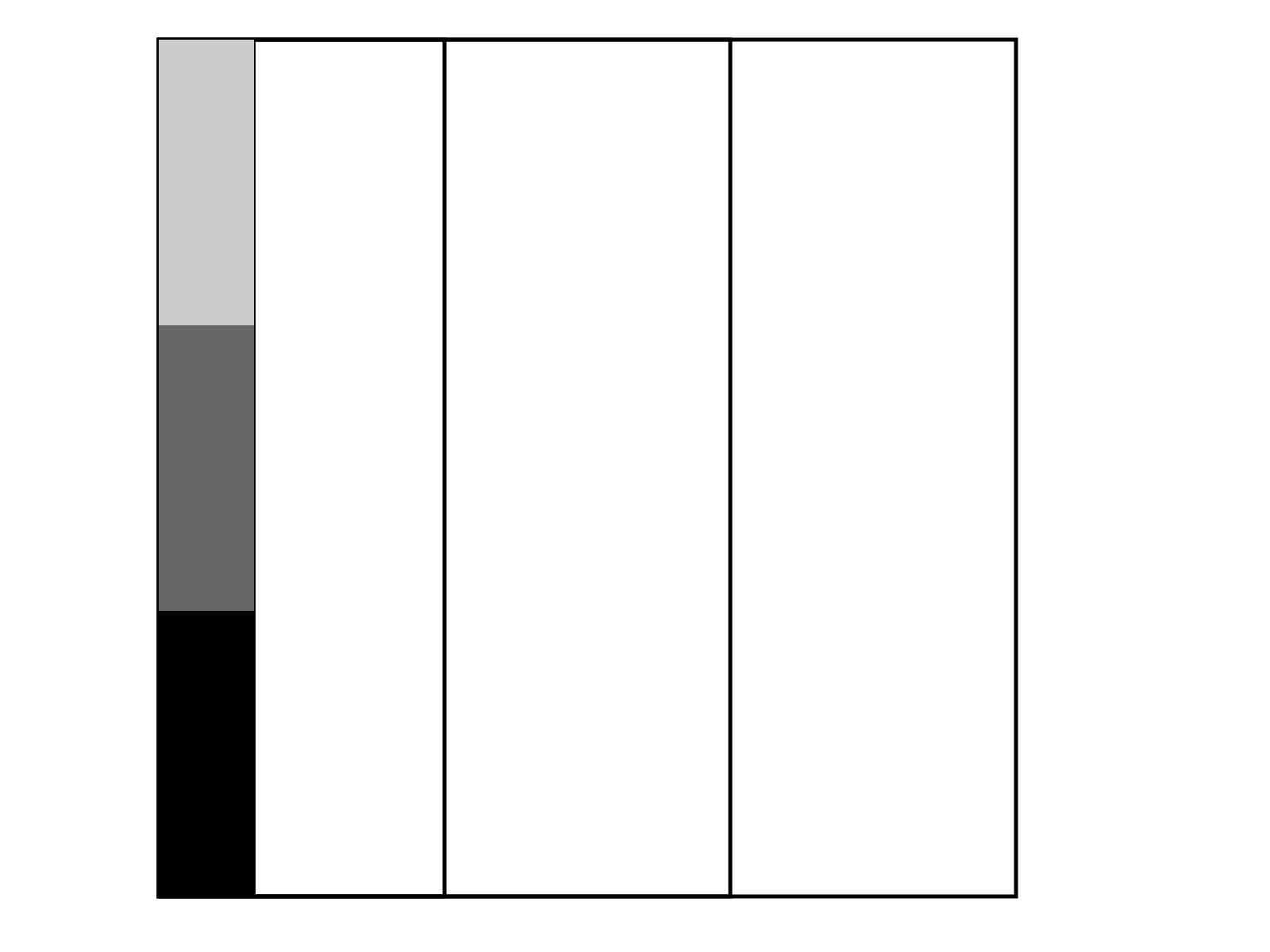}} $\xrightarrow{\tilde{h}_{1,n+1}}$ \raisebox{-2cm}{\includegraphics[scale=.1]{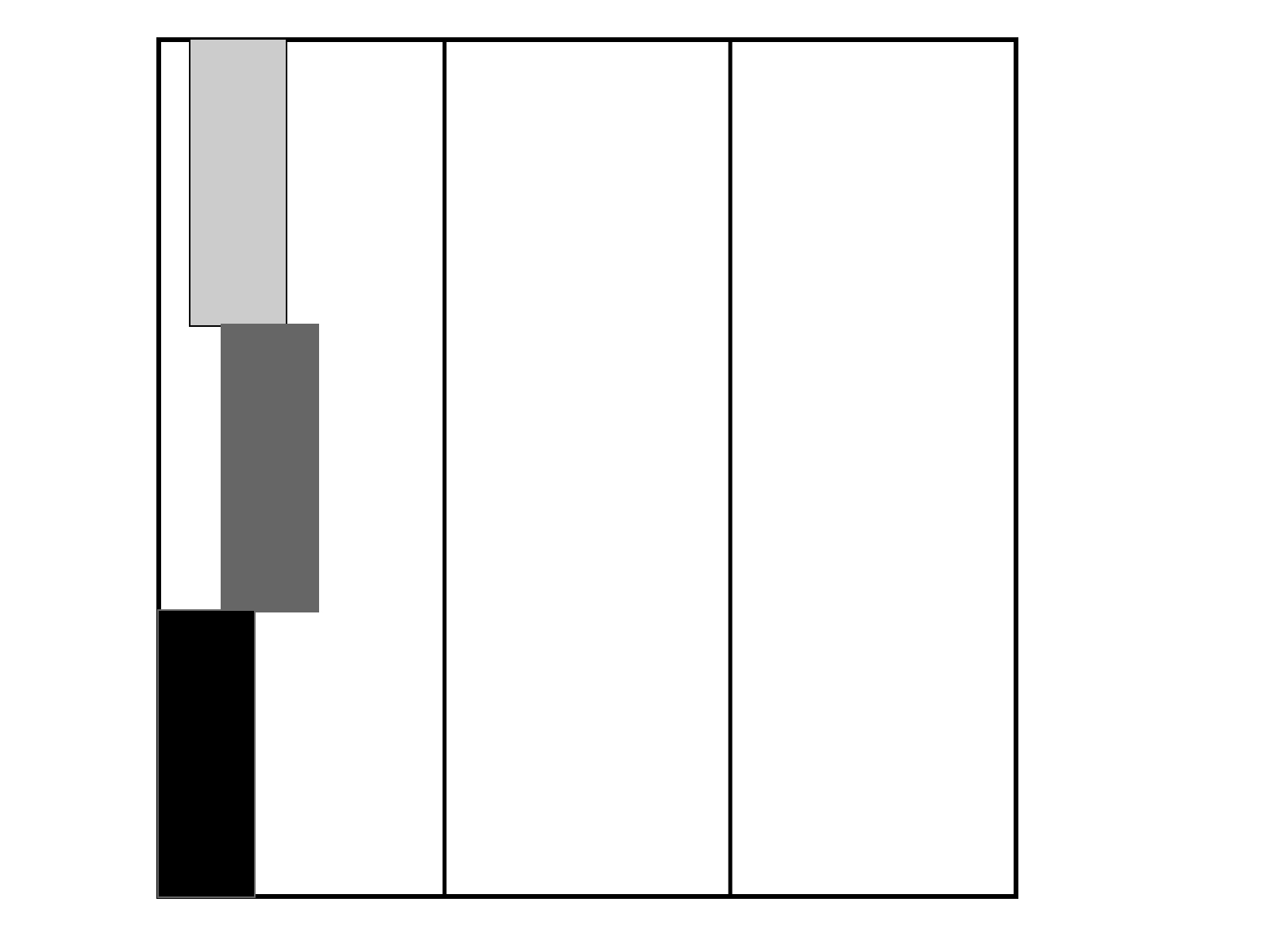}}\\
& \\
& \\
$\xrightarrow{\tilde{h}_{2,n+1}}$ & \raisebox{-2cm}{\includegraphics[scale=.1]{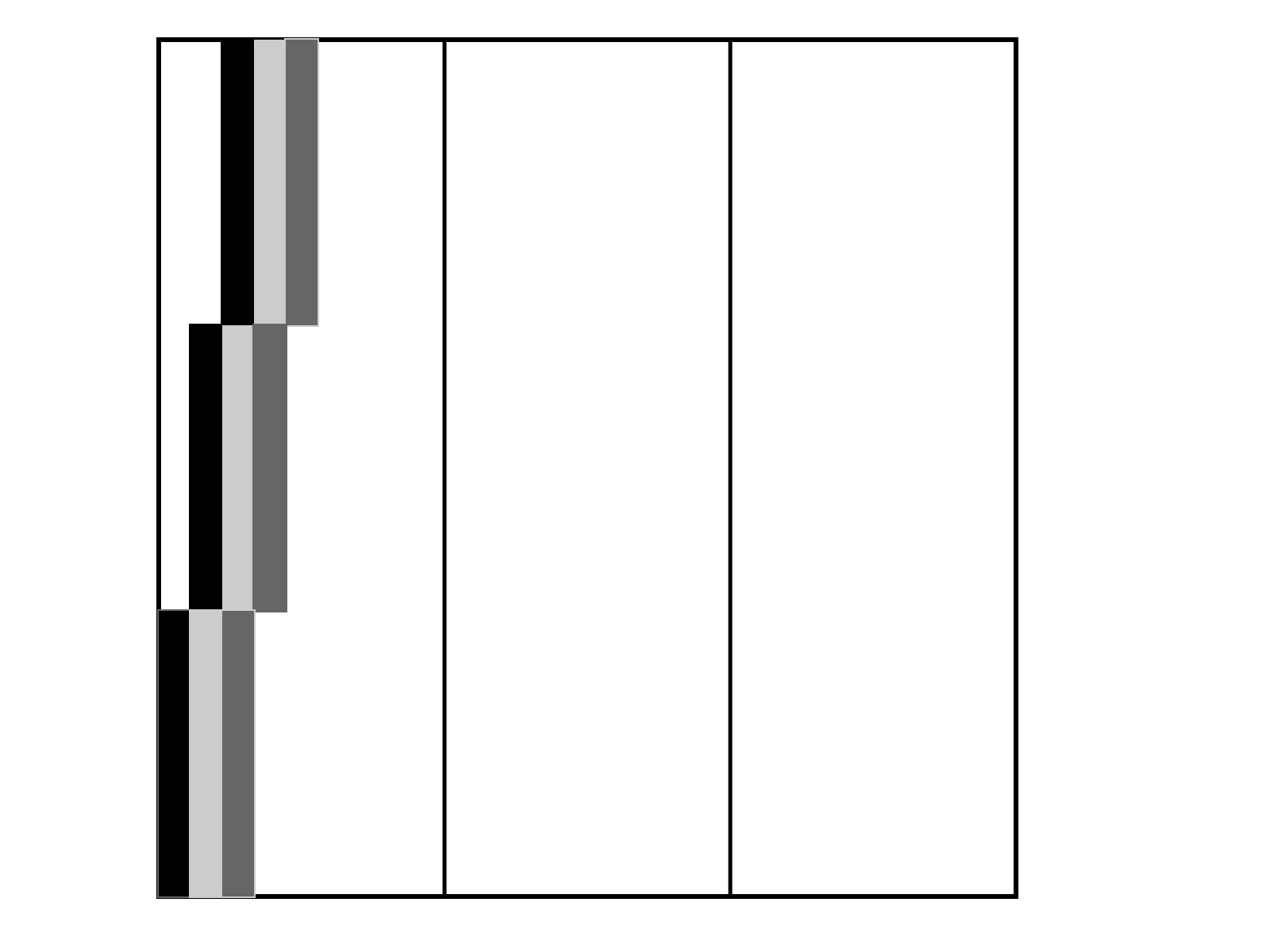}} $\xrightarrow{\tilde{h}_{3,n+1}}$ \raisebox{-2cm}{\includegraphics[scale=.1]{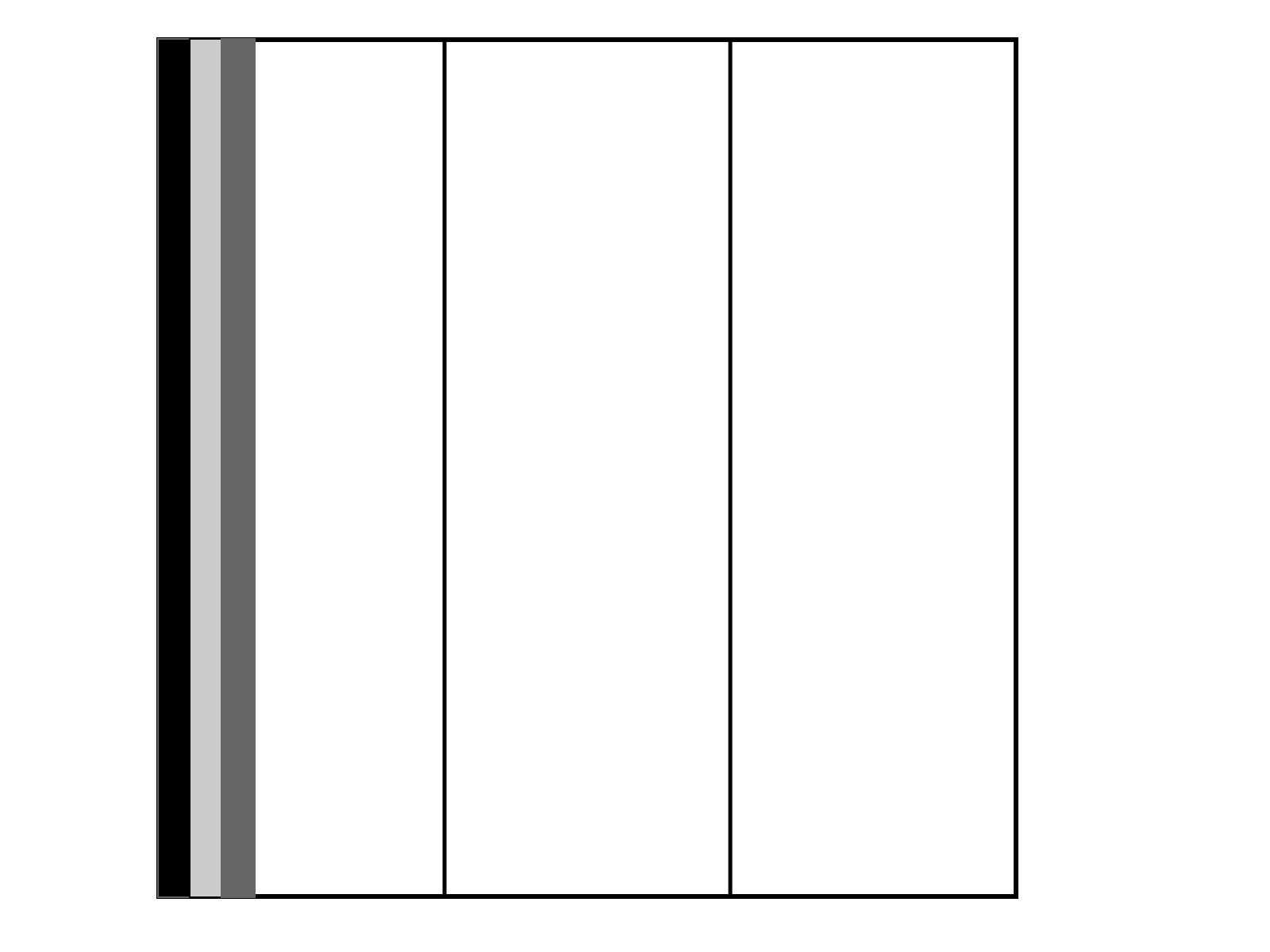}}\\
& \\
& \\
$\xrightarrow{\tilde{h}_{4,n+1}}$ & \raisebox{-2cm}{\includegraphics[scale=.1]{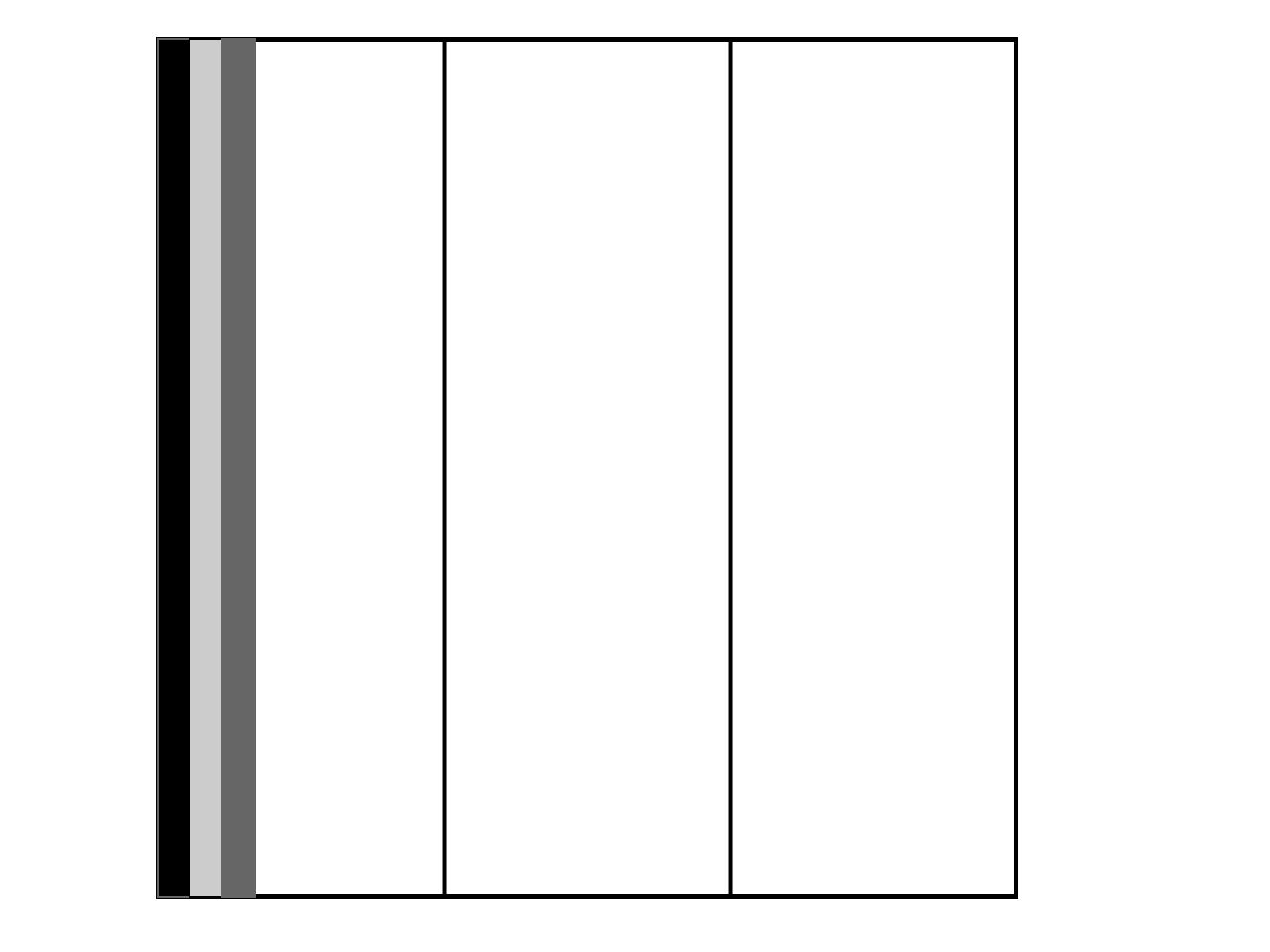}} $\qquad\quad $ \raisebox{-2cm}{\includegraphics[scale=.1]{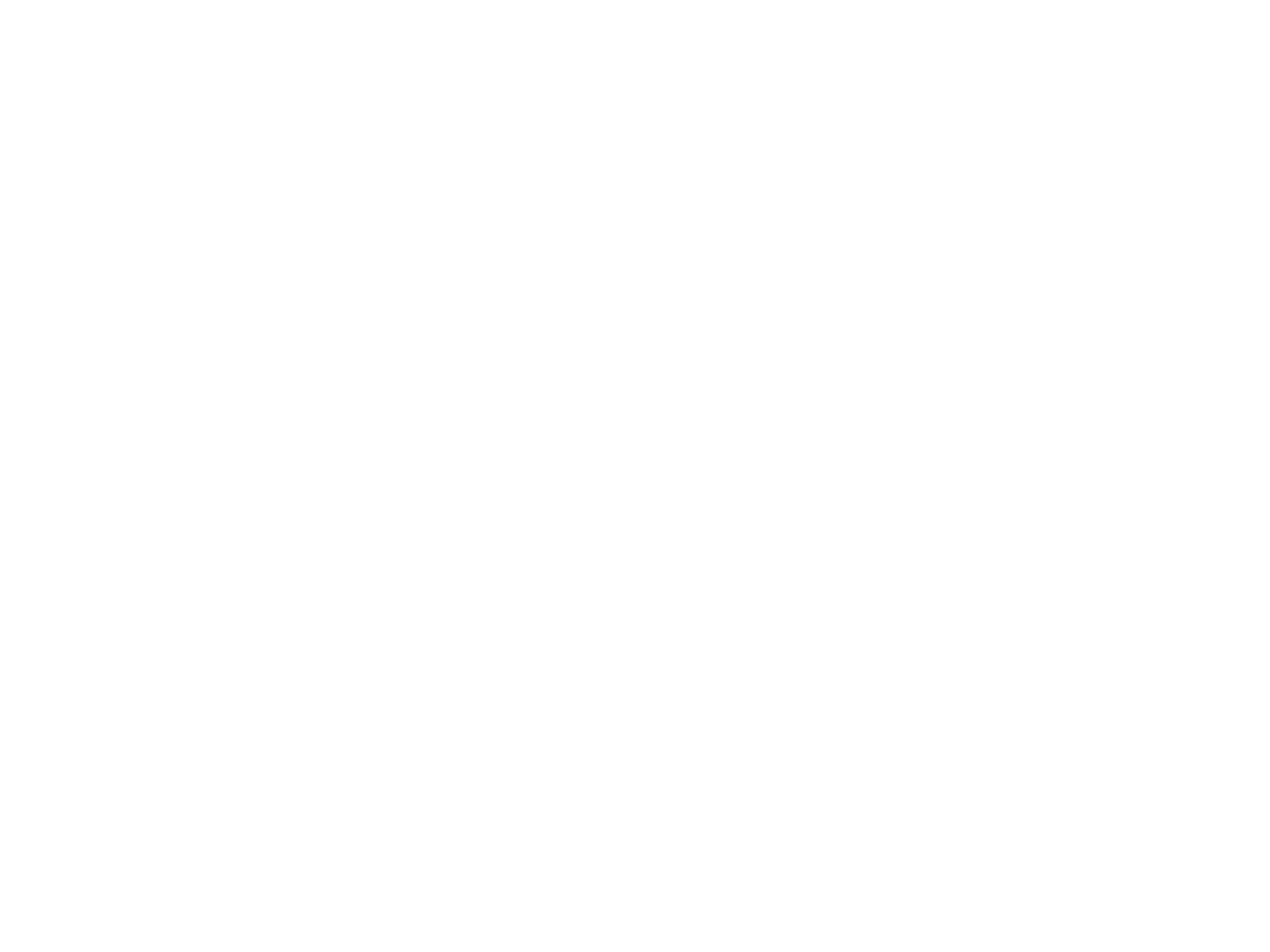}}\\
& \\
&
\end{tabular}
\caption{Illustration of the action of $\tilde{h}_{n+1}$ on three vertically stacked up atoms of $\mathcal{G}_{l,q}$  and how it is transformed into three atoms of $\mathcal{T}_q$ by $\tilde{h}_{n+1}$. This diagram is drawn with $q_n=3$ and $l_n=3$. Note that the choice of $k_n$ and $s_n$ has not been used for this illustration since they do not alter the combinatorics in any way. Also one should note that the last map $\tilde{h}_{4,n+1}$ preserves the atoms of $\mathcal{T}_{q}$, so combinatorially it is unnecessary.}
\label{combinatorics}
\end{figure}

Here our goal is to construct a discontinuous map $\tilde{h}_{n+1}:\T^2\to\T^2$ which ``breaks" down a partition of $\T^2$ and rearranges it into a generating partition. Later in this section, we will see that we can approximate this function by real-analytic diffeomorphisms.

Let us consider the following three step functions:
\begin{align}
&\tilde{\psi}_{1,n+1}:[0,1)\to \R &\text{ defined by}\quad &\tilde{\psi}_{1,n+1}(x)=\sum_{i=1}^{l_n-1}\frac{l_n-i}{l_n^2q_n}\chi_{[\frac{i}{l_n},\frac{i+1}{l_n}]}(x)\nonumber\\
&\tilde{\psi}_{2,n+1}:[0,1)\to \R &\text{ defined by}\quad &\tilde{\psi}_{2,n+1}(x)=\sum_{i=0}^{l_n^2q_n-1}\(\frac{i}{l_n}-\Big\lfloor\frac{i}{l_n}\Big\rfloor\)\chi_{[\frac{i}{l_n^2q_n},\frac{i+1}{l_n^2q_n}]}(x)\nonumber\\
&\tilde{\psi}_{3,n+1}:[0,1)\to \R &\text{ defined by}\quad &\tilde{\psi}_{3,n+1}(x)=\sum_{i=0}^{l_n-1}\frac{i}{l_n^2q_n}\chi_{[\frac{i}{l_n},\frac{i+1}{l_n}]}(x)\label{tilde psi}
\end{align}
Now we are ready to define the required map $\tilde{h}_{n+1}$ as the composition of four different maps:
\begin{align}
\tilde{h}_{n+1}:=\tilde{h}_{4,n+1}\circ \tilde{h}_{3,n+1}\circ \tilde{h}_{2,n+1}\circ \tilde{h}_{1,n+1}
\end{align}
Here the maps $\tilde{h}_{i,n+1}:\T^2\to\T^2$ for $i=1,2,3$ and $4$ are defined as follows:
\begin{align}
\tilde{h}_{1,n+1}((x_1,x_2))&=(x_1+\tilde{\psi}_{1,n+1}(x_2) \mod 1\; ,\; x_2)\nonumber\\
\tilde{h}_{2,n+1}((x_1,x_2))&=(x_1\; ,\; x_2+\tilde{\psi}_{2,n+1}(x_1) \mod 1)\nonumber\\
\tilde{h}_{3,n+1}((x_1,x_2))&=(x_1-\tilde{\psi}_{3,n+1}(x_2) \mod 1\; ,\; x_2)\nonumber\\
\tilde{h}_{4,n+1}((x_1,x_2))&:=(x_1,\; x_2+q_nx_2 \mod 1)
\end{align}

In order to properly understand the action of $\tilde{h}_{n+1}$ on $\T^2$, we look at the following two partitions of $\T^2$:
\begin{align}
\mathcal{T}_{q}&=\big\{\Delta_{i,q}:=[\frac{i}{q},\frac{i+1}{q})\times [0,1): i=0,1,\ldots, q-1\big\}\label{T}\\
\mathcal{G}_{l,q}&=\big\{[\frac{i}{lq},\frac{i+1}{lq})\times [\frac{j}{l},\frac{j+1}{l}): i=0,1,\ldots lq-1,\;\; j=0,1,\ldots l-1\big\}\label{G}
\end{align}
Now the action of $\tilde{h}_{n+1}^{-1}$ is clear. Its job is to break the atoms of the partition $\mathcal{T}_{l_n^2q_n}$ into $l_n$ pieces and rearrange those broken pieces into the atoms of the partition $\mathcal{G}_{l_n,q_n}$. Note that the diameter of the atoms of partition $\mathcal{G}_{l_n,q_n}$ is bounded by a small number, namely $\frac{1}{l}$. See figure \ref{combinatorics} for an elementary demonstration of this combinatorics at work. 

Of course $\tilde{h}_{n+1}$ is not analytic and in fact it is not even continuous, so we need to construct $h_{n+1}$, an entire diffeomorphism `close' to $\tilde{h}_{n+1}$. At this point we need to remind ourself that this will be possible because we can approximate step functions by real-analytic functions which extends holomorphically to the entire complex plane as an entire function. This $h_{n+1}$ will somewhat successfully emulate the combinatorial picture produced by $\tilde{h}_{n+1}$ i.e. $h_{n+1}^{-1}$ successfully compresses the bulk of the measure of an element of $\mathcal{T}_{q_{n+1}}$ into a set of small diameter but one needs to remember that there is an `error' set (see \ref{error set}) of small measure where we loose control.

\subsubsection*{Construction of the real-analytic conjugating diffeomorphism}

\begin{figure}
\centering
\begin{tabular}{c c c c}
\includegraphics[scale=.26]{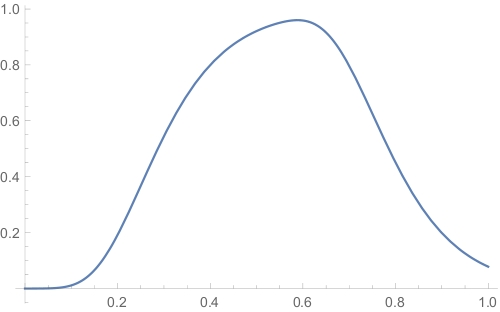} & \includegraphics[scale=.26]{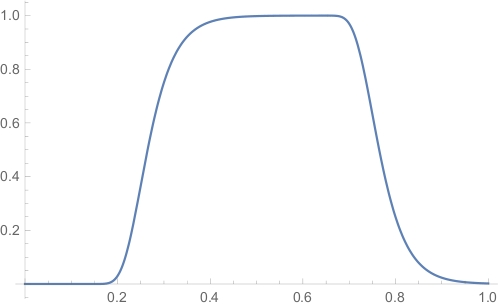} & \includegraphics[scale=.26]{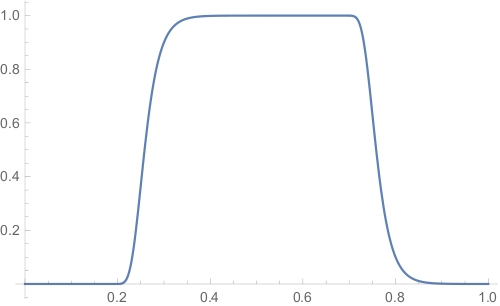} & \includegraphics[scale=.26]{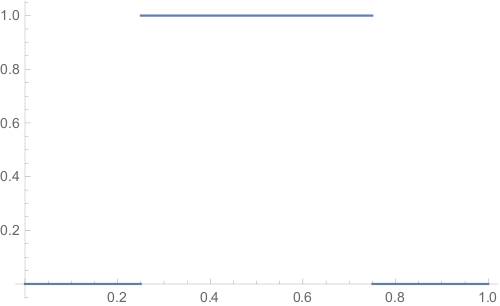}
\end{tabular}
\caption{Illustration (using mathematica) of approximation of a step function $\chi_{[\frac{1}{4},\frac{3}{4}]}$ using the function $s$. Here the first three pictures depict the function $s(x):=\exp(-\exp(-A(x-\frac{1}{4})))-\exp(-\exp(-A(x-\frac{3}{4})))$ is depicted with $A=10$, $A=25$ and $A=45$. We should also note that $s$ here is not periodic. In Lemma \ref{approx}, we use an infinite sum of such functions to obtain periodicity.}
\label{approximation of step functions}
\end{figure}

We need the following lemma about approximation of step functions by real-analytic functions. This will allow us to construct an analytic $h_{n+1}$ `close' to the discontinuous map $\tilde{h}_{n+1}$ obtained by using step functions. The approximation used in the lemma is illustrated in Figure \ref{approximation of step functions}.

\begin{lemma}\label{approx}
Let $k$ and $N$ be two positive integer and $\a=(\a_0,\ldots,\a_{k-1})\in [0,1)^k$. Consider a step function of the form 
\begin{align*}
\tilde{s}_{\a,N}:[0,1)\to \R\quad\text{ defined by}\quad \tilde{s}_{\a,N}(x)=\sum_{i=0}^{kN-1}\tilde{\a}_i\chi_{[\frac{i}{kN},\frac{i+1}{kN})}(x)
\end{align*}
Here $\tilde{\a}_i:=\a_j$ where $j:=i\mod k$. Then, given any $\e>0$ and $\d>0$, there exists a $\frac{1}{N}$-periodic entire function $s_{\a,N}:\R\to\R$ satisfying 
\begin{align}\label{nearness}
\sup_{x\in[0,1)\setminus F}|s_{\a,N}(x)-\tilde{s}_{\a,N}(x)|<\e
\end{align}
Where $F=\cup_{i=0}^{kN-1}I_i\subset [0,1)$ is a union of intervals centred around $\frac{i}{kN},\;i=1,\ldots, kN-1$ and $I_0=[0,\frac{\d}{2kN}]\cup[1-\frac{\d}{2kN},1)$ and $\l(I_i)=\frac{\d}{kN}\;\forall\; i$.
\end{lemma}

\begin{proof}
We define the following function:
\begin{align*}
& s_{\a,N}:\R\to\R\quad\text{ defined by }\\
& s_{\a,N}(x)=\sum_{n=-\infty}^\infty\(\sum_{i=0}^{k-1}\a_i\big(e^{-e^{-A(x-\frac{nk+i}{kN})}}- e^{-e^{-A(x-\frac{nk+i+1}{kN})}}\big)\)
\end{align*}
Note that  $s_{\a,N}$ is a $\frac{1}{N}$-periodic entire function defined on the real line. After choosing a large enough constant $A$, we can guarantee that $s_{\a,N}$ satisfies condition \ref{nearness}.
\end{proof}



Now we are ready to produce an entire diffeomorphism `close' to $\tilde{h}_{n+1}$. First we use Lemma \ref{approx} to define three entire functions approximating the step functions defined above as follows:
\begin{align}
&\psi_{1,n+1}:=s_{\a^{(1)},N^{(1)}}\;\text{where }\a^{(1)}_0=0,\; \a^{(1)}_i=\frac{l_n-i}{l_n^2q_n},\; i=1,\ldots k_n-1,\; N^{(1)}=1;\\
&\psi_{2,n+1}:=s_{\a^{(2)},N^{(2)}}\;\text{where }\a^{(2)}_i=\frac{i}{l_n},\; i=1,\ldots l_n-1,\; N^{(2)}=l_nq_n;\\
&\psi_{3,n+1}:=s_{\a^{(3)},N^{(3)}}\;\text{where }\a^{(3)}_i=\frac{i}{l_n^2q_n},\; i=1,\ldots l_n-1,\; N^{(3)}=1;
\end{align}
With the above definitions, and using $3\e=\d=\frac{1}{2^{l_n^2q_n}l_n^2q_n}$ in Lemma \ref{approx}, we can conclude that there exists three measurable sets $F_{i,n+1}$ for $i=1,2,3$ such that:
\begin{align}
\sup_{x\in[0,1)-F_{1,n}}|\psi_{1,n+1}(x)-\tilde{\psi}_{1,n+1}(x)|<\e\quad\quad\l(F_{1,n})<\frac{1}{2^{l_n^2q_n}l_n^2q_n}\label{error F1}\\
\sup_{x\in[0,1)-F_{2,n}}|\psi_{2,n+1}(x)-\tilde{\psi}_{2,n+1}(x)|<\e\quad\quad\l(F_{2,n})<\frac{1}{2^{l_n^2q_n}l_n^2q_n}\label{error F2}\\
\sup_{x\in[0,1)-F_{3,n}}|\psi_{3,n+1}(x)-\tilde{\psi}_{3,n+1}(x)|<\e\quad\quad\l(F_{3,n})<\frac{1}{2^{l_n^2q_n}l_n^2q_n}\label{error F3}
\end{align}
Additionally its important to notice at this point that $\psi_{2,n+1}$ is $\frac{1}{q_n}$ periodic. 

Now we are ready to define $h_{n+1}$ as the composition of four entire diffeomorphisms: 
\begin{align}\label{h_n}
h_{n+1}:=h_{4,n+1}\circ h_{3,n+1}\circ h_{2,n+1}\circ h_{1,n+1}
\end{align}
Here the diffeomorphisms $h_{i,n+1}:\T^2\to\T^2$ for $i=1,2,3$ and $4$ are defined as follows:
\begin{align}
h_{1,n+1}((x_1,x_2))&:=(x_1+\psi_{1,n+1}(x_2) \mod 1\; ,\; x_2)\nonumber\\
h_{2,n+1}((x_1,x_2))&:=(x_1\;,\; x_2+\psi_{2,n+1}(x_1) \mod 1)\nonumber\\
h_{3,n+1}((x_1,x_2))&:=(x_1-\psi_{3,n+1}(x_2) \mod 1\; ,\;x_2)\nonumber\\
h_{4,n+1}((x_1,x_2))&:=(x_1,\; x_2+q_nx_2 \mod 1)
\end{align}
Note that it follows from the periodicity of $\psi_{i,n+1}$ (which was guaranteed by lemma \ref{approx}):  
\begin{align}\label{commute}
h_{n+1}\circ\phi^{\a_n}=\phi^{\a_n}\circ h_{n+1}
\end{align}
We define our conjugating diffemorphism $H_{n+1}$ as follows:
\begin{align}
& H_{n+1}:=h_{n+1}\circ H_n
\end{align}
This completes the construction of the conjugating diffeomorphism at the $n+1$ th stage of the induction. 

We end this section by defining,
\begin{align}\label{diophantine approx}
q_{n+1}:=k_nl_n^2q_n,\quad\quad p_{n+1}:=k_nl_n^2p_n+1,\quad\quad \a_{n+1}:=\frac{p_{n+1}}{q_{n+1}}
\end{align}
Finally, we let
\begin{align}
& T_{n+1}:=H_{n+1}^{-1}\circ\phi^{\a_{n+1}}\circ H_{n+1}\label{Tn}
\end{align}

\begin{remark}
Note that 
\begin{align*}
T_{n+1} & =H_{n+1}^{-1}\circ\phi^{\a_{n+1}}\circ H_{n+1}\\
& = h_{1,n+1}^{-1}\circ h_{2,n+1}^{-1}\circ h_{3,n+1}^{-1}\circ h_{4,n+1}^{-1}\circ\phi^{\a_{n+1}}\circ h_{4,n+1}\circ h_{3,n+1}\circ h_{2,n+1}\circ h_{1,n+1}\\
& = h_{1,n+1}^{-1}\circ h_{2,n+1}^{-1}\circ \phi^{\a_{n+1}}\circ  h_{2,n+1}\circ h_{1,n+1}
\end{align*}
So $h_{3,n+1}$ and $h_{4,n+1}$ are redundant for the definition of $T_{n+1}$, but $h_{3,n+1}$ will later play an important role as a tool to ``untwist" our construction albeit with an error. $h_{4,n+1}$ will be crucial to ensure that this error set we would construct does not contain any whole orbit. This is important because we wish to conclude that our limiting diffeomorphism is uniquely ergodic and hence we do not wish to give up control over even a single orbit.
\end{remark}


\section{Existence of a limiting diffeomorphism conjugate to a rotation}

The purpose of this section is to show that the sequence $T_n$ defined in \ref{Tn} connverges to some $T\in\text{Diff }^\omega_\rho(\T^2)$ for any $\rho$. In fact this $T$ will be entire.

We define the \emph{error set} $E_{q_{n+1}}\subset \T^2$ as follows:
\begin{align}\label{error set}
H_{n+1}(E_{q_{n+1}}):=E_{q_{n+1}}^{(v)}\bigcup E_{q_{n+1}}^{(d)}
\end{align}
Here
\begin{align*}
& E_{q_{n+1}}^{(v)}:=\bigcup_{i=0}^{l_n^2q_n-1}\Big[\frac{i}{l_n^2q_n}-\frac{1}{2^{l_n^2q_n}l_n^2q_n},\frac{i}{l_n^2q_n}+\frac{1}{2^{l_n^2q_n}l_n^2q_n}\Big] \times \T^1 \\
& E_{q_{n+1}}^{(d)}=h_{4,n+1}\(\bigcup_{i=0}^{l_n-1}\T^1\times \Big[\frac{i}{l_n}-\frac{1}{2^{l_n^2q_n}l_n^2q_n},\frac{i}{l_n}+\frac{1}{2^{l_n^2q_n}l_n^2q_n}\Big]\)
\end{align*}
Note that $E_{q_{n+1}}$ is the set where we do not have any control. The darkened region of figure $3$ represents how this error set is spread over $\T^2$.

Getting back to the proof, we define the following partition of $\T^2$:
\begin{align}
 \mathcal{F}_q:=\{H_{n+1}^{-1}\Delta_{i,q}:i=0,1,\ldots,q-1\}
\end{align}
We need the following proposition about some basic inclusions and estimates.

\begin{figure}\label{es}
\centering
\includegraphics[scale=.1]{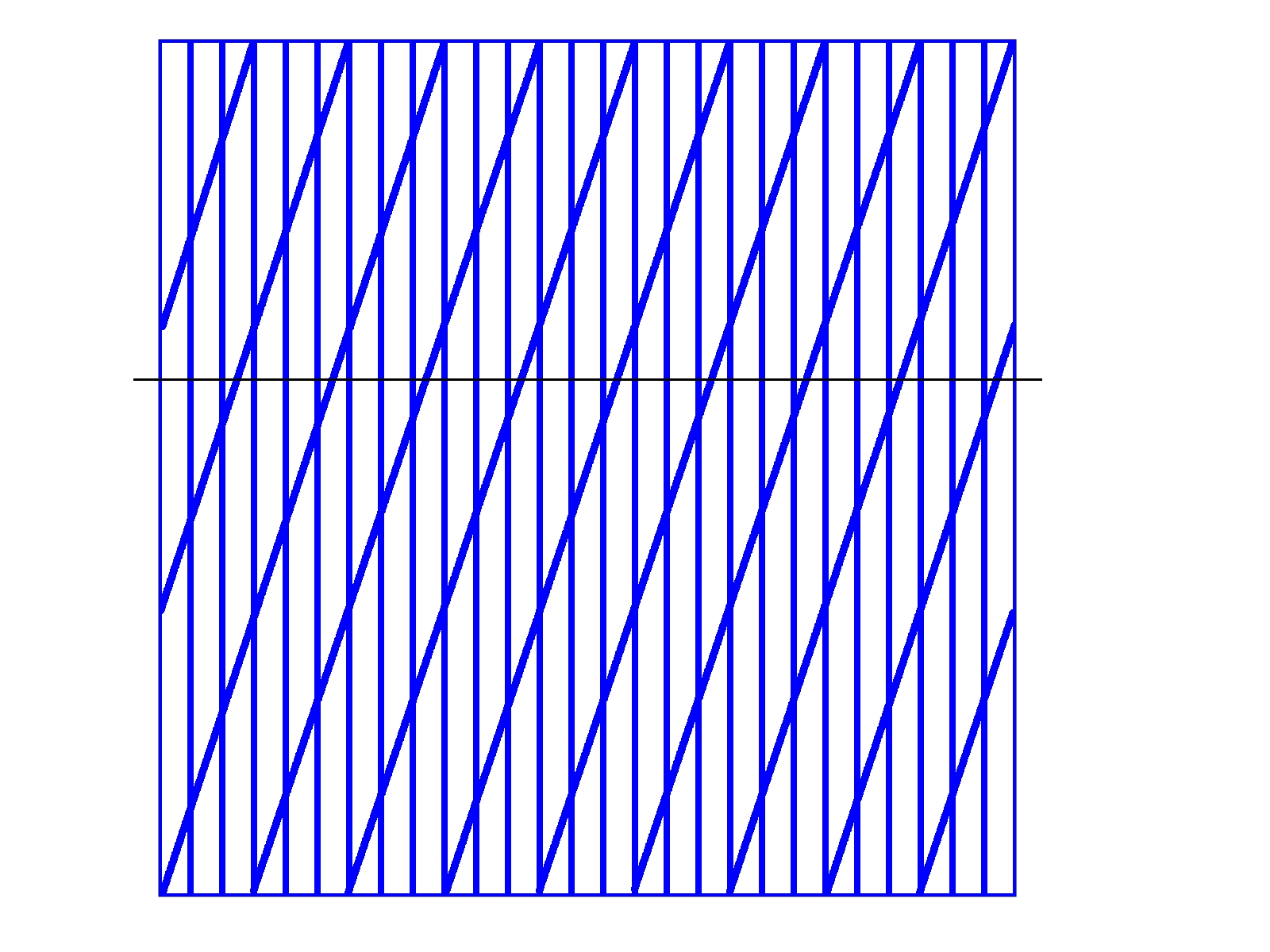} 
\caption{The darkened region represents $H_{n+1}(E_{q_{n+1}})$ and the horizontal line represents a $\phi^t$-orbit. The diagram is made with $q_n=3$ and $l_n=3$. Note that the diagonal parts corresponds to $E^{(d)}_{q_{n+1}}$ and the vertical portion corresponds to $E^{(v)}_{q_{n+1}}$.}
\end{figure}

\begin{proposition}\label{conditions}
\begin{align*}
(1)&\quad \mu(E_{q_{n+1}})\leq\frac{2}{2^{l_n^2q_n}}\\
(2)&\quad \mathcal{F}_{q_{n+1}}(x)\setminus E_{q_{n+1}}\subset \mathcal{F}_{q_n}(x)\quad\forall\; x\in\T^2\setminus E_{q_{n+1}}\\
(3)&\quad \mathcal{F}_{q_n}\to\e\text{ as } n\to\infty\qquad(\text{Here }\e\text{ is the trivial partition into points}).\\
(4)&\quad T_n\big(\mathcal{F}_{q_n}\big)=\mathcal{F}_{q_n} 
\end{align*}
\end{proposition}

\begin{proof}
Statement (1) follows as a result of direct but obvious computation. 

To conclude statement (2), first we note that $F_{1,n}\times\T^1\subset E_{q_{n+1}}^{(v)}$, $\T^1\times F_{2,n}\subset h_{4,n+1}^{-1}E_{q_{n+1}}^{(d)}$ and $F_{3,n}\times\T^1\subset E_{q_{n+1}}^{(v)}$ (see \ref{error F1}, \ref{error F2} and \ref{error F3}).
So, for any $0\leq i<q_{n+1}$, $i'=\lfloor\frac{i}{l_n^2q_n}\rfloor$ and $j=\lfloor\frac{i}{q_n}\rfloor$, we have:
\begin{align*}
& \quad h_{n+1}^{-1}(\Delta_{i',l_n^2q_n}\setminus H_{n+1}(E_{q_{n+1}})) \subset  \Delta_{j,q_n}\\
\Rightarrow & \quad h_{n+1}^{-1}(\Delta_{i,q_{n+1}}\setminus H_{n+1}(E_{q_{n+1}})) \subset  \Delta_{j,q_n}\\
\Rightarrow & \quad H_{n+1}^{-1}(\Delta_{i,q_{n+1}}\setminus H_{n+1}(E_{q_{n+1}})) \subset H_n^{-1}(\Delta_{j,q_n})\\
\Rightarrow & \quad H_{n+1}^{-1}(\Delta_{i,q_{n+1}})\setminus E_{q_{n+1}} \subset H_n^{-1}(\Delta_{j,q_n})
\end{align*}

Statement (3) follows from the observation that 
\begin{align*}
&\quad \text{diam }(h_{n+1}^{-1}\Delta_{i,q_{n+1}}\cap (\T^2\setminus H_{n+1}(E_{q_{n+1}})))<\frac{1}{l_n}\\
\Rightarrow &\quad \text{diam }(H_{n+1}^{-1}\Delta_{i,q_{n+1}}\cap (\T^2\setminus H_n^{-1}\circ H_{n+1}(E_{q_{n+1}})))<\|DH_n\|\frac{1}{l_n}<\frac{1}{2^n}\qquad(\text{see \ref{ln}})
\end{align*}
Note that $\lim_{n\to\infty}\mu(\T^2\setminus H_n^{-1}\circ H_{n+1}(E_{q_{n+1}}))=1$.

Proof of statement (4) is obvious.
\end{proof}

Our next objective is to show that $T_n$ converges in the real-analytic topology to a real-analytic diffeomorphism $T$. We have the following proposition in that direction:

\begin{proposition}\label{conv}
\begin{align*}
T_n\to T
\end{align*}
Where $T$ is real-analytic.
\end{proposition}

\begin{proof}
Here we show that if $k_n$ is chosen appropriately large enough then $\{T_n\}$ is Cauchy in $\text{Diff }^{\omega}_\rho(\T^2,\mu)$. This will allow us to conclude that $T_n$ converges to some real-analytic $T$.

Fix any $\rho>0$ (since we are dealing with entire functions, any choice of $\rho$ will work). First we note that 
\begin{align*}
&\|h_{n+1}^{-1}\circ\phi^{\frac{1}{k_nl_n^2q_n}}\circ h_{n+1}-Id\|_\rho \\
&=\max\Big\{\inf_{m\in\Z}\big\{\sup_{(z_1,z_2)\in\Omega_\rho}\{\psi_{1,n+1}(z_2)-\psi_{1,n+1}\big(z_2+\psi_{2,n+1}(z_1+\psi_{1,n+1}(z_2))\\
& \quad\quad\quad\quad\quad\quad\quad -\psi_{2,n+1}(z_1+\psi_{1,n+1}(z_2)+\frac{1}{k_nl_n^2q_n})\big)+\frac{1}{k_nl_n^2q_n}+m\}\big\},\\
& \quad\quad\quad\quad\inf_{m\in\Z}\big\{\sup_{(z_1,z_2)\in\Omega_\rho}\{\psi_{2,n+1}(z_1+\psi_1(z_2))\\
& \quad\quad\quad\quad\quad\quad\quad -\psi_2(z_1+\psi_1(z_2)+\frac{1}{k_nl_n^2q_n})+m\}\big\}\Big\}
\end{align*}
So if we choose $k_n$ to be a large enough integer, we can make the above norm arbitrarily small. More precisely, we have the following inequality:
\begin{align}\label{k_n restriction}
\|h_{n+1}^{-1}\circ\phi^\frac{1}{k_nl_n^2q_n}\circ h_{n+1}-I\|_\rho<\frac{1}{2^n\|DH_n\|_\rho}
\end{align}
We should also point out that we will make one more demand on the largeness of $k_n$ in the proof of proposition \ref{longest}. So we proceed by assuming $k_n$ to be larger than requirements imposed by both the restrictions. We complete the proof using the following computation involving the mean value theorem and estimate \ref{k_n restriction}:
\begin{align*}
\|T_{n+1}-T_n\|_\rho=&\|H_{n+1}^{-1}\circ \phi^{\a_{n+1}}\circ H_{n+1}-H_n^{-1}\circ\phi^{\a_n}\circ H_n\|_\rho\\
\leq &\|DH_n\|_\rho\|h_{n+1}^{-1}\circ\phi^{\a_{n+1}}\circ h_{n+1}-\phi^{\a_n}\|_\rho\\
= &\|DH_n\|_\rho\|h_{n+1}^{-1}\circ\phi^{\a_{n}}\circ\phi^\frac{1}{k_nl_n^2q_n}\circ h_{n+1}-\phi^{\a_n}\|_\rho\quad(\text{see }\ref{diophantine approx})\\
= &\|DH_n\|_\rho\|\phi^{\a_n}\circ h_{n+1}^{-1}\circ\phi^\frac{1}{k_nl_n^2q_n}\circ h_{n+1}-\phi^{\a_n}\|_\rho\quad(\text{see }\ref{commute})\\
= &\|DH_n\|_\rho\|h_{n+1}^{-1}\circ\phi^\frac{1}{k_nl_n^2q_n}\circ h_{n+1}-I\|_\rho\\
\leq &\frac{1}{2^n}
\end{align*}
This implies that the sequence $\{T_n\}$ Cauchy and hence converges to $T$ which is real-analytic.
\end{proof}

\subsubsection*{Existence of an ergodic diffeomorphism metrically conjugate to a circle rotation}

First we approximate an irrational rotation of the circle by rational rotations. Let $\a_n$ and be as before and consider a sequence of partitions of the circle as follows:
\begin{align*}
\mathcal{C}_{q_n}:=\Big\{\Gamma_{i,q_n}:=\Big[\frac{i}{q_n},\frac{i+1}{q_n}\Big):\; i=0,1,\ldots q_n-1\Big\}
\end{align*}
Clearly this is a sequence of partitions are monotonic and generating. We also, define a sequence of maps:
\begin{align*}
R_{\a_n}:S^1\to S^1,\quad\quad \text{ defined by }x\mapsto x+\a_n
\end{align*}
So, we have $R_{\a_n}\to R_\a$.
We also define 
\begin{align*}
\tilde{E}_{q_{n+1}}:=\bigcup_{i=0}^{q_{n+1}}\Big[\frac{i}{l_n^2q_n}-\frac{\mu(E_{q_{n+1}})}{2l_n^2q_n},\frac{i}{l_n^2q_n}+\frac{\mu(E_{q_{n+1}})}{2l_n^2q_n}\Big] 
\end{align*}

Following the notation of lemma \ref{mtl} we let $M^{(1)}:=\T^2,\mu^{(1)}:=\mu,\mathcal{P}^{(1)}_n:=\mathcal{F}_{q_n}, E^{(1)}_n:=E_{q_n}, M^{(2)}:=\T^1,\mu^{(2)}:=\l, \mathcal{P}^{(2)}_{n}:=\mathcal{C}_{q_n}$ and $E^{(2)}_n:=\tilde{E}_{q_{n+1}}$. Finally we define the conjugacy $K_n$ by $K_n(H_n^{-1}\Delta_{i,q_n})=\Gamma_{i,q_n}$. This along with proposition \ref{conditions} and proposition \ref{conv} together gives us that $\mathcal{P}^{(i)}_n$ is generating and  conditions \ref{mtl 1}, \ref{mtl 2}, \ref{mtl 4} and \ref{mtl 5} in lemma \ref{mtl}. Conditions \ref{mtl 7} and \ref{mtl 8} follows from the definition. Now note that $\phi^{\alpha_n}$ preserves $E_{q_{n+1}}^{(v)}$ and $E_{q_{n+1}}^{(d)}$ and hence $T_{n+1}$ preserves $E_{q_{n+1}}$. This gives us \ref{mtl 6} and we have essentially proved the following proposition:

\begin{proposition}
$T$ constructed above is a real-analytic diffeomorphisms of the two dimensional torus $\T^2$ preserving the Lebesgue measure and it is metrically isomorphic to an irrational rotation of the circle. 
\end{proposition}


\section{Proof of unique ergodicity}

Of course it follows from the previous discussion that the limiting diffeomorphism $T$ constructed above is ergodic with respect to the Lebesgue measure. We would like to show that $T$ is in fact uniquely ergodic. The proof we present here is almost identical to the one presented in section 3 of \cite{FSW}. However marginal changes needs to be made because our error set is more spread out than the one constructed in \cite{FSW}. This is crucial to ensure that we do not loose any finite orbits even though we have control over every $\T^1$ orbit. 

We need the following abstract lemma from \cite{FSW} for the proof of unique ergodicity:

\begin{lemma}\label{convcond}
Let $q_n$ be an increasing sequence of natural numbers and $T_n:(M,\mu)\to (M,\mu)$ be a sequence of transformations which converge uniformly to a transformation $T$. Suppose that for each continuous function $g$ from a dense set of continuous functions $\Phi$ there is a constant $c$ such that 
\begin{align*}
&\frac{1}{q_n}\sum_{i=0}^{q_n-1}g\circ T^i_n(x)\xrightarrow{n\to\infty} c\qquad\text{uniformly}
\end{align*}
and,
\begin{align*}
&d^{(q_n)}(T_n,T):=\max_{x\in M}\big(\max_{0\leq i\leq q_n} d(T^i_n x, T^i x)\big) \to 0
\end{align*}
Then $T$ is uniquely ergodic.
\end{lemma}  

We need the following very important estimate before we can prove unique ergodicity.

\begin{proposition}\label{longest}
Given any $\e>0$ and any $(\frac{1}{l_n},\e)$-uniformly continuous function\footnote{A function $g:(X,d)\to(Y,d')$ is called $(\delta,\e)$-uniformly continuous if $g(B_\delta(x))\subset B_\e(g(x))\;\;\forall x\in X$.} $g:\T^2\to\R$, we have for all $x\in \T^2$,
\begin{align*}
\quad \Big|\frac{1}{q_{n+1}}\sum_{j=0}^{q_{n+1}-1}& g\circ h_{n+1}^{-1}\circ \phi^{\frac{j}{q_{n+1}}}(x)-
\int_{\T^2}gd\mu\Big|<2\e+\frac{6\|g\|_0}{n^2}\\
\end{align*}
\end{proposition}

\begin{proof}
Fix any $x_0\in\T^2$. Then $x=\phi^\frac{i_0}{q_{n+1}}x_0\in \Delta_{i,q_{n+1}}$ for some $i_0$ and $i$ such that $\Delta_{i,q_{n+1}}\cap E_{q_{n+1}}^{(v)}=\emptyset$ and $x\not\in E_{q_{n+1}}^{(d)}$. Then for any $y\in \Delta_{i,q_{n+1}}\cap \big(H_{n+1}(E_{q_{n+1}})\big)^c$, it follows from our construction,
\begin{align*}
&\quad d\big(h_{n+1}^{-1}(x), h_{n+1}^{-1}(y)\big)<\frac{1}{l_n}
\end{align*}
Using the $(\frac{1}{l_n},\e)$-uniform continuity of $g$,
\begin{align*}
 |g\circ h_{n+1}^{-1}(x)-g\circ h_{n+1}^{-1}(y)|<2\e
\end{align*}
After averaging over all $y\in \Delta_{i,q_{n+1}}\cap \big(H_{n+1}(E_{q_{n+1}})\big)^c$
\begin{align}\label{1234}
\Big|g & \circ h_{n+1}^{-1}(x)-
\frac{1}{\mu\big(\Delta_{i,q_{n+1}}\cap (H_{n+1}(E_{q_{n+1}}))^c\big)}\int_{h_{n+1}^{-1}\big(\Delta_{i,q_{n+1}}\cap (H_{n+1}(E_{q_{n+1}}))^c\big)} gd\mu\Big|<2\e
\end{align}
Now consider the set $\Lambda^{(x)}=\{j:\Delta_{j,q_{n+1}}\cap E_{q_{n+1}}^{(v)}=\emptyset,$ and $\phi^\frac{j}{q_{n+1}}(x)\not\in E^{(d)}_{q_{n+1}};\; 0\leq j< q_{n+1}\}$. Also let $L^{(x)}=\text{card }\{\Lambda^{(x)}\}$. We note that 
\begin{align}\label{lx estimate}
L^{(x)}\geq q_{n+1}(1-\frac{2}{k_n})
\end{align}
Getting back to \ref{1234} and summing over $\Lambda^{(x)}$ we obtain:
\begin{align*}
&\Big|\frac{1}{q_{n+1}}\sum_{j\in \Lambda^{(x)}} g\circ h_{n+1}^{-1}\circ \phi^{\frac{j}{q_{n+1}}}(x)-
\(\frac{1}{q_{n+1}}\sum_{j\in\Lambda^{(x)}}\frac{1}{\mu\big(\Delta_{j,q_{n+1}}\cap (H_{n+1}(E_{q_{n+1}}))^c\big)}\cdot\\
&\qquad\qquad\qquad\qquad\qquad\qquad\qquad\qquad\int_{h_{n+1}^{-1}\big(\Delta_{j,q_{n+1}}\cap (H_{n+1}(E_{q_{n+1}}))^c\big)} gd\mu\)\Big|<2\e
\end{align*}
Using the fact that $E_{q_{n+1}}$ is a set of small measure,
\begin{align*}
\Big|\frac{1}{q_{n+1}} & \sum_{j\in \Lambda^{(x)}} g\circ h_{n+1}^{-1}\circ \phi^{\frac{j}{q_{n+1}}}(x)-\\
&\quad\frac{1}{q_{n+1}}\sum_{j\in\Lambda^{(x)}}\frac{1}{\mu(\Delta_{j,q_{n+1}})}\int_{h_{n+1}^{-1}\big(\Delta_{j,q_{n+1}}\cap (H_{n+1}(E_{q_{n+1}}))^c\big)} gd\mu\Big|<2\e+\frac{\|g\|_0}{q_{n+1}}
\end{align*}
Using the fact that $q_{n+1}\mu(\Delta_{j,q_{n+1}})=1$,
\begin{align*}
\Big|\frac{1}{q_{n+1}}\sum_{j\in \Lambda^{(x)}} g\circ h_{n+1}^{-1}\circ \phi^{\frac{j}{q_{n+1}}}(x)-\sum_{j\in\Lambda^{(x)}}\int_{h_{n+1}^{-1}\big(\Delta_{j,q_{n+1}}\cap (H_{n+1}(E_{q_{n+1}}))^c\big)} gd\mu\Big|<2\e+\frac{\|g\|_0}{q_{n+1}}
\end{align*}
Changing the index of the second sum and compensating for the error,
\begin{align*}
\Big|\frac{1}{q_{n+1}}\sum_{j\in \Lambda^{(x)}}g\circ h_{n+1}^{-1}\circ \phi^{\frac{j}{q_{n+1}}}(x)-\sum_{j=0}^{q_{n+1}-1}&\int_{h_{n+1}^{-1}\big(\Delta_{j,q_{n+1}}\cap (H_{n+1}(E_{q_{n+1}}))^c\big)} gd\mu\Big|\\&<2\e+\frac{\|g\|_0}{q_{n+1}}+\frac{\|g\|_0(q_{n+1}-L^{(x)})}{q_{n+1}}
\end{align*}
Using estimate \ref{lx estimate} and once again observing that $E_{q_{n+1}}$ is a set of small measure (see proposition \ref{conditions}),
\begin{align*}
\Big|\frac{1}{q_{n+1}}\sum_{j\in \Lambda^{(x)}}& g\circ h_{n+1}^{-1}\circ \phi^{\frac{j}{q_{n+1}}}(x)-
\int_{\T^2}gd\mu\Big|<2\e+\frac{\|g\|_0}{q_{n+1}}+\frac{2\|g\|_0}{k_n}+\frac{2\|g\|_0}{2^{l_n^2q_n}}
\end{align*}
We note that $2^{l_n^2q_n}>n^2$. And once again using estimate \ref{lx estimate} and changing the indexing of the first sum,
\begin{align*}
\Big|\frac{1}{q_{n+1}}\sum_{j=0}^{q_{n+1}-1}& g\circ h_{n+1}^{-1}\circ \phi^{\frac{j}{q_{n+1}}}(x)-
\int_{\T^2}gd\mu\Big|<2\e+\frac{4\|g\|_0}{k_n}+\frac{2\|g\|_0}{n^2}\\
\end{align*}
The proposition follows if we assume $k_n>n^2$.
\end{proof}

\begin{proposition}
The limiting diffeomorphism $T$ obtained in proposition \ref{conv} is uniquely ergodic.
\end{proposition}

\begin{proof}
Let $G=\{g_k\}$ be a dense set of Lipshitz functions dense in $C^0(\T^2,\R)$. Let $K_n$ be a uniform Lipshitz constants for $g_1\circ H_{n}^{-1},\ldots,g_n\circ H_{n}^{-1}$. Then from proposition \ref{longest} with $\e=\frac{1}{n^2}$ and $N$ such that $l_N>n^2K_n$, we get for each $1\leq k\leq n$,
\begin{align*}
\Big|\frac{1}{q_{N+1}}\sum_{i=0}^{q_{N+1}-1}g_k\circ H_{N+1}^{-1}\circ \phi^{\frac{i}{q_{N+1}}}(x)-\int_{\T^2} g_k\circ H_{N+1}^{-1} d\mu\Big|<\frac{2}{n^2}+\frac{6\|g_k\|_0}{N^2}
\end{align*}
Using the fact that $H_{N+1}$ is measure preserving and reordering the orbit as necessary, we get,
\begin{align*}
\Big|\frac{1}{q_{N+1}}\sum_{i=0}^{q_{N+1}-1}g_k\circ T^i(x)-\int_{\T^2} g_k d\mu\Big|<\frac{2}{n^2}+\frac{6\|g_k\|_0}{N^2}
\end{align*}
Taking $n\to\infty$, the first condition in Lemma \ref{convcond} is satisfied. For the second condition one can check that:
\begin{align*}
d^{(q_n)}(T_{n+1},T_n)<\frac{1}{2^n}
\end{align*}
Note that here we map need to choose a $k_n$ larger that the one required in \ref{k_n restriction} since we want closeness of $h_{n+1}^{-1}\circ\phi^\frac{i}{k_nl_n^2q_n}\circ h_{n+1}$ to $I$ for all $i\leq q_n$. So the second condition of the lemma is also satisfied.
\end{proof}


\section{Sketch of the construction for higher dimensional torus} \label{section_high_dimension}

Here we point out that our construction can easily be generalized to a higher dimensional torus. More precisely, for $d=1,2,3,\dots$, we can obtain a real-analytic diffeomorphism $f_d:\T^d\to\T^d$ that is uniquely ergodic with respect to the Lebesgue measure and metrically isomorphic to an irrational rotation of the circle. The case where $d=1$ is just the irrational rotation itself and complete proofs for the $d=2$ case has been given in this article. Now let us assume $d\geq 3$ and we will try to give a description of the combinatorics that would allow a generalization of the present case.

Let $\phi$ be a $\T^1$ action on $\T^d$ by translation on the first coordinate. More precisely,
\begin{align*}
\phi^t:\T^d\to\T^d\quad\quad \text{defined by}\quad\phi^t:(x_1,x_2,\ldots,x_d)\mapsto (x_1+t,x_2,\ldots,x_d)
\end{align*}
Our target is to construct a real analytic diffeomorphism $h_{n+1}^{(d)}:\T^d\to\T^{(d)}$. $h_{n+1}^{(d)}$ will be constructed close to a discontinuous map $\tilde{h}_{n+1}^{(d)}$ which maps a partition of small diameter ($\mathcal{G}_{l_n,q_n}^{(d)}$) to a partition permuted cyclically by the $\T^1$ action ($\mathcal{T}_{l_n^{d}q_n}^{(d)}$). These two partitions are defined as follows:
\begin{align*}
\mathcal{T}_{l_n^dq_n}^{(d)}&=\big\{\Delta_{i,l_n^dq_n}:=[\frac{i}{l_n^dq_n},\frac{i+1}{l_n^dq_n})\times \T^{(d-1)}: i=0,1,\ldots, l_n^dq_n-1\big\}\\
\mathcal{G}_{l_n,q_n}^{(d)}&=\big\{[\frac{i_1}{l_nq_n},\frac{i_1+1}{l_nq_n})\times[\frac{i_2}{l_n},\frac{i_2+1}{l_n})\times\ldots\times [\frac{i_d}{l_n},\frac{i_d+1}{l_n}): \\
& \qquad\qquad\qquad\qquad\qquad\qquad\qquad i_1=0,1,\ldots l_nq_n-1,\;(i_2,\ldots, i_d)\in \{0,1,2,\ldots, l_n-1\}^d\big\}
\end{align*}
We construct $\tilde{h}_{n+1}^{(d)}$ by a finite induction argument on dimension. Note that for $d=2$ we put $\mathcal{T}^{(2)}_{l_nq_n}=\mathcal{T}_{l_nq_n}$, $\tilde{h}_{n+1}^{(2)}=\tilde{h}_{n+1}$ and observe that $\mathcal{G}_{l_n,q_n}^{(2)}=(\tilde{h}_{n+1}^{(2)})^{-1}\mathcal{T}_{l_n,q_n}^{(2)}$ (see \ref{T} and \ref{G}). This completes the construction of for $d=2$ and also starts the induction.

Next we assume that the construction has been carried out upto $d-1$ and we construct $\tilde{h}_{n+1}^{(d)}$. Before we do that, observe that 
\begin{align*}
((L_n^{-1}\circ\tilde{h}_{n+1}^{d-1}\circ L_n)^{-1}\times I) (\mathcal{T}_{l_n^dq_n}^{(d)}) & =\big\{[\frac{i_1}{l_n^2q_n},\frac{i_1+1}{l_n^2q_n}]\times[\frac{i_2}{l_n},\frac{i_2+1}{l_n}]\times\ldots\times [\frac{i_{d-1}}{l_n},\frac{i_{d-1}+1}{l_n}]\times \T^1: \\
& \quad i_1=0,1,\ldots l_n^2q_n-1,\;(i_2,\ldots, i_{d-1})\in \{0,1,2,\ldots, l_n-1\}^{d-1}\big\}
\end{align*}
Here $I$ is the identity map on $\T^1$ and $L_n$ is multiplication in the first coordinate by $l_n$. So we observe that we have the structure of the partition $\mathcal{G}^{(d)}_{l_n,q_n}$ apart for the last coordinate. Now we consider the following three maps from $\T^d\to\T^d$:
\begin{align*}
\tilde{h}_{1,n+1}^{(d)}((x_1,\ldots, x_d))&:=(x_1+\tilde{\psi}_{1,n+1}(x_d) \mod 1\; ,\ldots,\; x_d)\\
\tilde{h}_{2,n+1}^{(d)}((x_1,\ldots, x_d))&:=(x_1\; ,\ldots,\; x_d+\tilde{\psi}_{2,n+1}(x_1) \mod 1)\\
\tilde{h}_{3,n+1}^{(d)}((x_1,\ldots, x_d))&:=(x_1-\tilde{\psi}_{3,n+1}(x_d) \mod 1\; ,\ldots,\; x_d)
\end{align*}
See \ref{tilde psi} for definition of $\tilde{\psi}_{i,n+1}$. Observe that $\tilde{h}_{3,n+1}^{(d)}\circ \tilde{h}_{2,n+1}^{(d)}\circ \tilde{h}_{1,n+1}^{(d)}$
maps $\mathcal{G}^{(d)}_{l_n,q_n}$ to the partition $(L_n^{-1}\circ (\tilde{h}_{n+1}^{d-1})^{-1}\circ L_n\times I) (\mathcal{T}_{l_n^dq_n}^{(d)})$. Hence we define 
\begin{align*}
\tilde{h}^{(d)}_{n+1}=\tilde{h}_{3,n+1}^{(d)}\circ \tilde{h}_{2,n+1}^{(d)}\circ \tilde{h}_{1,n+1}^{(d)}\circ (L_n^{-1}\circ (\tilde{h}_{n+1}^{d-1})^{-1}\circ L_n\times I)
\end{align*}
So, $\tilde{h}^{(d)}_{n+1}$ constructed above maps the partition $\mathcal{G}^{(d)}_{l_n,q_n}$ to the partition $\mathcal{T}^{(d)}_{l_n^dq_n}$. However this map is discontinuous. One can use entire functions to approximate the step functions (in the sense of lemma \ref{approx}) that went in the definition of $\tilde{h}^{(d)}_{n+1}$ and construct a real-analytic diffeomorphism $h^{(d)}_{n+1}$ `close' to $\tilde{h}^{(d)}_{n+1}$. Our required real-analytic diffeomorphism $f_d$ can then be obtained after taking a limit of the conjugates in the real-analytic topology.

\vspace{1cm}

\noindent \emph{Acknowledgements.} The author would like to thank Anatole Katok for introducing the problem and highly appreciates his support by numerous useful discussions and encouragement. The author would also thank Changguang Dong for  many helps he provided. Finally, the author extends his gratitude towards Philipp Kunde for carefully going through an initial draft of this paper and pointing out a critical error. He also recommended a correction which worked and is incorporated in the present version.


\end{document}